\documentclass[reqno,12pt]{amsart}
\usepackage{amscd,amsfonts,amssymb}
\numberwithin{equation}{section}
\newtheorem{Theorem}{Theorem}[section]
\newtheorem{Lemma}{Lemma}[section]
\newtheorem{Corollary}{Corollary}[section]

\theoremstyle{definition}

\theoremstyle{remark}
\newtheorem{Remark}{Remark}[section]
\newtheorem{Example}{Example}[section]

\newcommand{\essinf}{\mathop{\rm ess \, inf}\limits}
\newcommand{\esssup}{\mathop{\rm ess \, sup}\limits}

\newcommand{\diver}{\mathop{\rm div}\nolimits}
\author{Andrej A. Kon'kov}
\address{Department of Differential Equations,
Faculty of Mechanics and Mathematics,
Mo\-s\-cow Lo\-mo\-no\-sov State University,
Vorobyovy Gory,
119992 Moscow, Russia}
\email{konkov@mech.math.msu.su}
\title[]{On solutions of quasilinear elliptic inequalities 
containing terms with lower-order derivatives}
\thanks{The research was supported by RFBR, grant 11-01-12018-ofi-m-2011.}
\date{}
\begin{document}

\maketitle

\section{Introduction}

Let $\Omega$ be an unbounded open subset of ${\mathbb R}^n$, $n \ge 2$.
As in~\cite{LU}, by ${W_{p, loc}^1 (\Omega)}$ we mean 
the space of measurable functions that belong to ${W_p^1 (B_r \cap \Omega)}$ 
for all real numbers $r > 0$ satisfying the condition
$B_r \cap \Omega \ne \emptyset$,
where $B_r$ is the open ball in ${\mathbb R}^n$ 
of radius $r$ and center at zero.
The space ${L_{\infty, loc} (\Omega)}$ is defined in a similar way.

Consider the inequality
\begin{equation}
	{\rm div} \, A (x, D u)
	+
	b (x) |D u|^{p - 1}
	\ge
	q (x) g (u)
	\quad
	\mbox{in }
	\Omega,
	\label{1.1}
\end{equation}
where 
$D = (\partial / \partial x_1,\ldots,\partial / \partial x_n)$
is the gradient operator,
$b \in L_{\infty, loc} (\Omega)$, and
$A : \Omega \times {\mathbb R}^n \to {\mathbb R}^n$
is a measurable function such that
$$
	C_1
	|\zeta|^p
	\le
	\zeta
	A (x, \zeta),
	\quad
	|A (x, \zeta)|
	\le
	C_2
	|\zeta|^{p-1}
$$
with some constants
$C_1 > 0$,
$C_2 > 0$,
and
$p > 1$
for almost all
$x \in \Omega$
and for all
$\zeta \in {\mathbb R}^n$.

We assume that $S_r \cap \Omega \ne \emptyset$ for all $r > r_0$,
where $r_0 > 0$ is some real number and $S_r$ is the sphere in ${\mathbb R}^n$ 
of radius $r$ and center at zero. 
Also let
$q \in L_{\infty, loc} (\Omega)$
and
$g \in C ([0, \infty))$
be non-negative functions and, moreover,
$g (t) > 0$ for all $t > 0$.
We denote
$$
	f_\sigma (r)
	=
	\frac{
		\essinf_{
			\Omega_{r / \sigma, \sigma r}
		}
		q
	}{
		1
		+
		r
		\,
		\esssup_{
			\Omega_{r / \sigma, \sigma r}
		}
		|b|
	},
	\quad
	r > r_0,
	\:
	\sigma > 1,
$$
$$
	q_\sigma (r)
	=
	\essinf_{
		\Omega_{r / \sigma, \sigma r}
	}
	q,
	\quad
	r > r_0,
	\:
	\sigma > 1
$$
and
$$
	g_\theta (t)
	=
	\inf_{
		(t / \theta, \theta t)
	}
	g,
	\quad
	t > 0,
	\:
	\theta > 1,
$$
where 
$
	\Omega_{r_1, r_2}
	=
	\{
		x \in \Omega : r_1  < |x| < r_2
	\},
$
$0 \le r_1 < r_2 \le \infty$.

A non-negative function 
$
	u 
	\in 
	W_{p, loc}^1 (\Omega)
	\cap
	L_{\infty, loc} (\Omega)
$ 
is called a solution of
inequality~\eqref{1.1} if the map $x \mapsto A (x, D u)$ is measurable and
$$
	- \int_\Omega
	A (x, D u)
	D \varphi
	\, dx
	+
	\int_\Omega
	b (x) |D u|^{p - 1}
	\varphi
	\, dx
	\ge
	\int_{
		\Omega
	}
	q (x) g (u)
	\varphi
	\, dx
$$
for any non-negative function
$
	\varphi 
	\in 
	C_0^\infty (
		\Omega
	).
$
As is customary, the condition
\begin{equation}
	\left.
		u
	\right|_{
		\partial \Omega
	}
	=
	0
	\label{1.2}
\end{equation}
means that
$
	\psi u
	\in
	{
		\stackrel{\rm \scriptscriptstyle o}{W}\!\!{}_p^1
		(
			\Omega
		)
	}
$
for any
$
	\psi
	\in
	C_0^\infty 
	(
		{\mathbb R}^n
	).
$
If $\Omega = {\mathbb R}^n$,
then~\eqref{1.2} is obviously valid for all
$
	u
	\in
	W_{p, loc}^1
	(
		{\mathbb R}^n
	).
$

For every solution of~\eqref{1.1}, \eqref{1.2} we put
$$
	M (r; u)
	=
	\esssup_{
		S_r
		\cap
		\Omega
	}
	u,
	\quad
	r > r_0,
$$
where the restriction of $u$ to
$
	S_r
	\cap
	\Omega
$
is understood in the sense of the trace and
the $\esssup$ in the right-hand side is with respect to
$(n-1)$-dimensional Lebesgue measure on $S_r$.

The research presented to your attention
deals with a priori estimates and blow-up
conditions for solutions of problem \eqref{1.1}, \eqref{1.2}.
The questions treated below were investigated 
mainly for nonlinearities of the Emden-Fowler type 
$g (t) = t^\lambda$~\cite{KL, MP, MPbook, Veron, VeronBook}.
Partial cases of inequality~\eqref{1.1} were also studied in~\cite{Keller, Osserman}.
In our paper, we consider the most general case.

\section{Main Results}

\begin{Theorem}\label{T2.1}
Let
\begin{equation}
	\int_1^\infty
	(g_\theta (t) t)^{- 1 / p}
	\,
	dt
	<
	\infty
	\label{T2.1.1}
\end{equation}
and
\begin{equation}
	\int_{r_0}^\infty
	(r f_\sigma (r))^{1 / (p - 1)}
	\,
	dr
	=
	\infty
	\label{T2.1.2}
\end{equation}
for some real numbers $\theta > 1$ and $\sigma > 1$,
then any solution of~\eqref{1.1}, \eqref{1.2} is trivial, 
i.e. $u = 0$ almost everywhere in $\Omega$.
\end{Theorem}

\begin{Remark}\label{R2.1}
We remind that, by definition, all solutions of~\eqref{1.1}, \eqref{1.2} are non-negative
functions since the domain of the function $g$ is the interval $[0, \infty)$.
\end{Remark}

\begin{Theorem}\label{T2.2}
Let there be real numbers $\theta > 1$ and $\sigma > 1$ such that~\eqref{T2.1.2}
is valid and, moreover, 
\begin{equation}
	\int_1^\infty
	(g_\theta (t) t)^{-1 / p}
	\,
	dt
	=
	\infty.
	\label{T2.2.1}
\end{equation}
Then any nontrivial solution of~\eqref{1.1}, \eqref{1.2} satisfies the estimate
$$
	\int_1^{
		M (r; u)
	}
	(g_\theta (t) t)^{-1 / p}
	\,
	dt
	\ge
	C
	\left(
		\int_{r_0}^r
		(\xi f_\sigma (\xi))^{1 / (p - 1)}
		\,
		d\xi
	\right)^{(p - 1) / p}
$$
for all sufficiently large $r$, 
where the constant $C > 0$ depends only on $n$, $p$, $\theta$, $\sigma$, $C_1$, and $C_2$.
\end{Theorem}

\begin{Theorem}\label{T2.3}
In the hypotheses of Theorem~$\ref{T2.2}$, let the condition
\begin{equation}
	\int_{r_0}^\infty
	\min
	\{
		(\xi f_\sigma (\xi))^{1 / (p - 1)},
		\,
		q_\sigma^{1 / p}
		(\xi)
	\}
	\,
	d\xi
	=
	\infty
	\label{T2.3.1}
\end{equation}
be valid instead of~\eqref{T2.1.2}.
Then any nontrivial solution of~\eqref{1.1}, \eqref{1.2} satisfies the estimate
\begin{align*}
	&
	\int_1^{
		M (r; u)
	}
	(g_\theta (t) t)^{-1 / p}
	\,
	dt
	+
	\int_1^{
		M (r; u)
	}
	g_\theta^{-1 / (p - 1)}
	(t)
	\,
	dt
	\\
	&
	\quad
	\ge
	C
	\int_{r_0}^r
	\min
	\{
		(\xi f_\sigma (\xi))^{1 / (p - 1)},
		\,
		q_\sigma^{1 / p}
		(\xi)
	\}
	\,
	d\xi
\end{align*}
for all sufficiently large $r$, 
where the constant $C > 0$ depends only on $n$, $p$, $\theta$, $\sigma$, $C_1$, and $C_2$.
\end{Theorem}

\begin{Theorem}\label{T2.4}
Let there be real numbers $\theta > 1$ and $\sigma > 1$ such that~\eqref{T2.1.1}
is valid and, moreover, 
$$
	\int_{r_0}^\infty
	(r f_\sigma (r))^{1 / (p - 1)}
	\,
	dr
	<
	\infty.
$$
Then any solution of~\eqref{1.1}, \eqref{1.2} satisfies the estimate
\begin{equation}
	\int_{
		M (r; u)
	}^\infty
	(g_\theta (t) t)^{-1 / p}
	\,
	dt
	\ge
	C
	\left(
		\int_r^\infty
		(\xi f_\sigma (\xi))^{1 / (p - 1)}
		\,
		d\xi
	\right)^{(p - 1) / p}
	\label{T2.4.1}
\end{equation}
for all sufficiently large $r$, 
where the constant $C > 0$ depends only on $n$, $p$, $\theta$, $\sigma$, $C_1$, and $C_2$.
\end{Theorem}

The proof of Theorems~\ref{T2.1}--\ref{T2.4} is given in Section~\ref{Proof}.
Now, we consider the case that
\begin{equation}
	|b (x)| \le \alpha |x|^k
	\label{2.1}
\end{equation}
with some constants $\alpha$ and $k$ for almost all $x \in \Omega$.

\begin{Corollary}\label{C2.1}
In formula~\eqref{2.1}, let $k \le -1$.
If there are real numbers $\theta > 1$ and $\sigma > 1$ such that~\eqref{T2.1.1} is valid and
\begin{equation}
	\int_{r_0}^\infty
	(r q_\sigma (r))^{1 / (p - 1)}
	\,
	dr
	=
	\infty,
	\label{C2.1.1}
\end{equation}
then any solution of~\eqref{1.1}, \eqref{1.2} is trivial. 
\end{Corollary}

\begin{Corollary}\label{C2.2}
In formula~\eqref{2.1}, let $k \le -1$.
Also suppose that conditions~\eqref{T2.2.1} and~\eqref{C2.1.1} are valid
for some real numbers $\theta > 1$ and $\sigma > 1$.
Then any nontrivial solution of~\eqref{1.1}, \eqref{1.2} satisfies the estimate
$$
	\int_1^{
		M (r; u)
	}
	(g_\theta (t) t)^{-1 / p}
	\,
	dt
	\ge
	C
	\left(
		\int_{r_0}^r
		(\xi q_\sigma (\xi))^{1 / (p - 1)}
		\,
		d\xi
	\right)^{(p - 1) / p}
$$
for all sufficiently large $r$, 
where the constant $C > 0$ depends only on 
$n$, $p$, $\theta$, $\sigma$, $C_1$, $C_2$, and $\alpha$.
\end{Corollary}

\begin{Corollary}\label{C2.3}
In the hypotheses of Corollary~$\ref{C2.2}$, let the condition
$$
	\int_{r_0}^\infty
	\min
	\{
		(\xi q_\sigma (\xi))^{1 / (p - 1)},
		\,
		q_\sigma^{1 / p}
		(\xi)
	\}
	\,
	d\xi
	=
	\infty
$$
be valid instead of~\eqref{C2.1.1}.
Then any nontrivial solution of~\eqref{1.1}, \eqref{1.2} satisfies the estimate
\begin{align*}
	&
	\int_1^{
		M (r; u)
	}
	(g_\theta (t) t)^{-1 / p}
	\,
	dt
	+
	\int_1^{
		M (r; u)
	}
	g_\theta^{-1  / (p - 1)}
	(t)
	\,
	dt
	\\
	&
	\quad
	\ge
	C
	\int_{r_0}^r
	\min
	\{
		(\xi q_\sigma (\xi))^{1 / (p - 1)},
		\,
		q_\sigma^{1 / p}
		(\xi)
	\}
	\,
	d\xi
\end{align*}
for all sufficiently large $r$, 
where the constant $C > 0$ depends only on 
$n$, $p$, $\theta$, $\sigma$, $C_1$, $C_2$, and $\alpha$.
\end{Corollary}

\begin{Corollary}\label{C2.4}
In formula~\eqref{2.1}, let $k \le -1$.
If there are real numbers $\theta > 1$ and $\sigma > 1$
such that~\eqref{T2.1.1} is valid and
$$
	\int_{r_0}^\infty
	(r q_\sigma (r))^{1 / (p - 1)}
	\,
	dr
	<
	\infty,
$$
then any solution of~\eqref{1.1}, \eqref{1.2} satisfies the estimate
$$
	\int_{
		M (r; u)
	}^\infty
	(g_\theta (t) t)^{-1 / p}
	\,
	dt
	\ge
	C
	\left(
		\int_r^\infty
		(\xi q_\sigma (\xi))^{1 / (p - 1)}
		\,
		d\xi
	\right)^{(p - 1) / p}
$$
for all sufficiently large $r$, 
where the constant $C > 0$ depends only on 
$n$, $p$, $\theta$, $\sigma$, $C_1$, $C_2$, and $\alpha$.
\end{Corollary}

Proof of Corollaries~\ref{C2.1}--\ref{C2.4} 
follows immediately from Teorems~\ref{T2.1}--\ref{T2.4}.
Really, if $k \le -1$ in~\eqref{2.1}, then 
$f_\sigma (r) \ge \gamma q_\sigma (r)$
for all $r > r_0$ and $\sigma > 1$,
where the constant $\gamma > 0$ depends only on $\alpha$, $k$, and $\sigma$.

In the case of $b = 0$, the above statements imply results of paper~\cite{MeDM}. 
Examples given in this paper demonstrate us the precision of Corollaries~\ref{C2.1}--\ref{C2.4}.

\begin{Corollary}\label{C2.5}
In formula~\eqref{2.1}, let $k > - 1$. 
If there are real numbers $\theta > 1$ and $\sigma > 1$ such that~\eqref{T2.1.1} is valid and
\begin{equation}
	\int_{r_0}^\infty
	(r^{-k} q_\sigma (r))^{1 / (p - 1)}
	\,
	dr
	=
	\infty,
	\label{C2.5.1}
\end{equation}
then any solution of~\eqref{1.1}, \eqref{1.2} is trivial. 
\end{Corollary}

\begin{proof}
The condition $k > - 1$ implies the inequality 
$f_\sigma (r) \ge \gamma r^{- k - 1} q_\sigma (r)$
for all sufficiently large $r$,
where the constant $\gamma > 0$ depends only on $\alpha$, $k$, and $\sigma$.
Thus, to complete the proof, it remains to use Theorem~\ref{T2.1}.
\end{proof}

\begin{Example}\label{E2.1}
Consider the inequality
\begin{equation}
	\diver (|D u|^{p-2} D u)
	+
	b (x) |D u|^{p-1}
	\ge
	q (x)
	u^\lambda
	\quad
	\mbox{in } {\mathbb R^n},
	\label{E2.1.1}
\end{equation}
where
$b \in L_{\infty, loc} ({\mathbb R^n})$ 
satisfies relation~\eqref{2.1} with $k > -1$ and 
$q \in L_{\infty, loc} ({\mathbb R^n})$
is a non-negative function such that
\begin{equation}
	q (x) \sim |x|^l
	\quad
	\mbox{as } x \to \infty,
	\label{E2.1.2}
\end{equation}
i.e. there exist constants $\alpha_1 > 0$ and $\alpha_2 > 0$ such that
$$
	\alpha_1 |x|^l \le q (x) \le \alpha_2 |x|^l
$$
for almost all $x$ in a neighborhood of infinity.

By Corollary~\ref{C2.5}, in the case of
\begin{equation}
	\lambda > p - 1
	\quad
	\mbox{and}
	\quad
	l \ge k - p + 1,
	\label{E2.1.3}
\end{equation}
any non-negative solution of~\eqref{E2.1.1} is trivial.

At the same time, if
$$
	\lambda > p - 1
	\quad
	\mbox{and}
	\quad
	l < k - p + 1,
$$
then 
$$
	u (x)
	=
	(
		\max
		\{
			|x|, 
			r_0 
		\}
	)^{
		(k - p + 1 - l) / (\lambda - p + 1)
	}
$$
is a positive solution of~\eqref{E2.1.1} 
for enough large $r_0$
with a non-negative function
$q \in L_{\infty, loc} ({\mathbb R^n})$ 
satisfying condition~\eqref{E2.1.2}
and a non-negative function
$b \in L_{\infty, loc} ({\mathbb R^n})$
such that
\begin{equation}
	b (x) \sim |x|^k
	\quad
	\mbox{as } x \to \infty.
	\label{E2.1.4}
\end{equation}

Using somewhat more complex reasoning, we can also show that~\eqref{E2.1.1} 
has a positive solution for all non-negative functions
$b \in L_{\infty, loc} ({\mathbb R^n})$ 
and
$q \in L_{\infty, loc} ({\mathbb R^n})$ 
if the first inequality in formula~\eqref{E2.1.3} does not hold.

Thus, both inequalities in~\eqref{E2.1.3} are exact.
\end{Example}

\begin{Example}\label{E2.2}
In~\eqref{E2.1.1}, let the non-negative function
$q \in L_{\infty, loc} ({\mathbb R^n})$
satisfy the relation
\begin{equation}
	q (x)
	\sim
	|x|^{k - p + 1}
	\log^\mu |x|
	\quad
	\mbox{as } x \to \infty.
	\label{E2.2.1}
\end{equation}
As in Example~\ref{E2.1}, we assume that~\eqref{2.1} is fulfilled with $k > -1$.

According to Corollary~\ref{C2.5}, if
\begin{equation}
	\lambda > p - 1
	\quad
	\mbox{and}
	\quad
	\mu \ge 1 - p,
	\label{E2.2.2}
\end{equation}
then any non-negative solution of~\eqref{E2.1.1} is trivial.

As noted above, the first inequality in formula~\eqref{E2.2.2} is exact.
Let us show that the second one is exact too.
In fact, if
$$
	\lambda > p - 1
	\quad
	\mbox{and}
	\quad
	\mu < 1 - p,
$$
then
$$
	u (x)
	=
	(
		\log
		\max
		\{
			|x|, 
			r_0 
		\}
	)^{
		(1 - p - \mu) / (\lambda - p + 1)
	}
$$
is a positive solution of~\eqref{E2.1.1}
for enough large $r_0 > 0$
with some non-negative functions
$b \in L_{\infty, loc} ({\mathbb R^n})$ 
and
$q \in L_{\infty, loc} ({\mathbb R^n})$
satisfying conditions~\eqref{E2.1.4} and~\eqref{E2.2.1}, respectively.
\end{Example}

\begin{Example}\label{E2.3}
Consider the inequality
\begin{equation}
	\diver (|D u|^{p-2} D u)
	+
	b (x) |D u|^{p-1}
	\ge
	q (x)
	u^{p - 1}
	\log^\lambda (1 + u)
	\quad
	\mbox{in } {\mathbb R^n},
	\label{E2.3.1}
\end{equation}
where
$b \in L_{\infty, loc} ({\mathbb R^n})$ 
satisfies relation~\eqref{2.1} with $k > -1$ and 
$q \in L_{\infty, loc} ({\mathbb R^n})$
is a non-negative function such that~\eqref{E2.1.2} holds.

By Corollary~\ref{C2.5}, the conditions
\begin{equation}
	\lambda > p
	\quad
	\mbox{and}
	\quad
	l \ge k - p + 1,
	\label{E2.3.2}
\end{equation}
imply that any non-negative solution of~\eqref{E2.3.1} is trivial.

On the other hand, if
$$
	\lambda > p
	\quad
	\mbox{and}
	\quad
	l < k - p + 1,
$$
then
$$
	u (x)
	=
	e^{
		(\max \{ |x|, r_0 \})^{
			(k - p + 1 - l) / (\lambda - p +1)
		}
	}
$$
is a positive solution of~\eqref{E2.3.1}
for enough large $r_0$
with some non-negative functions
$q \in L_{\infty, loc} ({\mathbb R^n})$ 
and
$b \in L_{\infty, loc} ({\mathbb R^n})$
satisfying relations~\eqref{E2.1.2} and~\eqref{E2.1.4}, respectively.
Therefore, the second inequality in formula~\eqref{E2.3.2} is exact.

The first inequality in~\eqref{E2.3.2} is also exact. 
Namely, in the case of $\lambda \le p$, it can be shown that~\eqref{E2.3.1} 
has a positive solution for all non-negative functions
$b \in L_{\infty, loc} ({\mathbb R^n})$ 
and
$q \in L_{\infty, loc} ({\mathbb R^n})$.
\end{Example}

\begin{Corollary}\label{C2.7}
Let~\eqref{2.1} be valid, where $k > -1$. 
Also suppose that conditions~\eqref{T2.2.1} and~\eqref{C2.5.1} hold
for some real numbers $\theta > 1$ and $\sigma > 1$.
Then any nontrivial solution of~\eqref{1.1}, \eqref{1.2} satisfies the estimate
$$
	\int_1^{
		M (r; u)
	}
	(g_\theta (t) t)^{- 1 / p}
	\,
	dt
	\ge
	C
	\left(
		\int_{r_0}^r
		(\xi^{- k} q_\sigma (\xi))^{1 / (p - 1)}
		\,
		d\xi
	\right)^{(p - 1) / p}
$$
for all sufficiently large $r$, 
where the constant $C > 0$ depends only on 
$n$, $p$, $\theta$, $\sigma$, $C_1$, $C_2$, $k$, and $\alpha$.
\end{Corollary}

\begin{proof}
We repeat the arguments given in the proof of Corollary~\ref{C2.5}
with Theorem~\ref{T2.1} replaced by Theorem~\ref{T2.2}.
\end{proof}

\begin{Example}\label{E2.5}
Consider inequality~\eqref{E2.1.1}, 
where $\lambda < p - 1$,
$b \in L_{\infty, loc} ({\mathbb R^n})$,
and $q \in L_{\infty, loc} ({\mathbb R^n})$ is a non-negative function.
We shall assume that condition~\eqref{2.1} is fulfilled for some $k > -1$.

Let~\eqref{E2.1.2} hold, where $l > k - p + 1$,
then in accordance with Corollary~\ref{C2.7} any nontrivial solution $u \ge 0$ of~\eqref{E2.1.1}
satisfies the estimate
\begin{equation}
	M (r; u)
	\ge
	C
	r^{
		(l - k + p - 1) / (p - 1 - \lambda)
	}
	\label{E2.5.1}
\end{equation}
for all enough large $r$,
where the constant $C > 0$ does not depend of $r$ and $u$.

If we replace~\eqref{E2.1.2} by~\eqref{E2.2.1},
where $\mu > 1 - p$, then~\eqref{E2.5.1} should be replaced by 
$$
	M (r; u)
	\ge
	C
	\log^{
		(\mu + p - 1) / (p - 1 - \lambda)
	}
	r
$$
for all enough large $r$,
where the constant $C > 0$ does not depend of $r$ and $u$.
\end{Example}

\begin{Example}\label{E2.6}
Assume that $u \ge 0$ is a nontrivial solution of the inequality
$$
	\diver (|D u|^{p-2} D u)
	+
	b (x) |D u|^{p-1}
	\ge
	q (x)
	u^\lambda
	\log^s (1 + u)
	\quad
	\mbox{in } {\mathbb R^n},
$$
where $\lambda < p - 1$, $s \in {\mathbb R}$,
$b \in L_{\infty, loc} ({\mathbb R^n})$ 
and, moreover, condition~\eqref{2.1} is valid for some $k > -1$.
Also let
$q \in L_{\infty, loc} ({\mathbb R^n})$
be a non-negative function such that~\eqref{E2.1.2} holds with $l > k - p + 1$.
Then, according to Corollary~\ref{C2.7}, we have
$$
	M (r; u)
	\ge
	C
	r^{
		(l - k + p -1) / (p - 1 - \lambda)
	}
	\log^{s / (p - 1 - \lambda)}
	r
$$
for all enough large $r$, where the constant $C > 0$ does not depend of $r$ and $u$.
\end{Example}

\begin{Corollary}\label{C2.8}
In the hypotheses of Corollary~$\ref{C2.7}$, let the condition
$$
	\int_{r_0}^\infty
	\min
	\{
		(\xi^{- k} q_\sigma (\xi))^{1 / (p - 1)},
		\,
		q_\sigma^{1 / p}
		(\xi)
	\}
	\,
	d\xi
	=
	\infty,
$$
be valid instead of~\eqref{C2.5.1}. 
Then any nontrivial solution of~\eqref{1.1}, \eqref{1.2} satisfies the estimate
\begin{align*}
	&
	\int_1^{
		M (r; u)
	}
	(g_\theta (t) t)^{-1 / p}
	\,
	dt
	+
	\int_1^{
		M (r; u)
	}
	g_\theta^{-1 / (p - 1)}
	(t)
	\,
	dt
	\\
	&
	\quad
	\ge
	C
	\int_{r_0}^r
	\min
	\{
		(\xi^{- k} q_\sigma (\xi))^{1 / (p - 1)},
		\,
		q_\sigma^{1 / p}
		(\xi)
	\}
	\,
	d\xi
\end{align*}
for all sufficiently large $r$, 
where the constant $C > 0$ depends only on 
$n$, $p$, $\theta$, $\sigma$, $C_1$, $C_2$, $k$, and $\alpha$.
\end{Corollary}

\begin{proof}
We repeat the arguments given in the proof of Corollary~\ref{C2.5}
with Theorem~\ref{T2.1} replaced by Theorem~\ref{T2.3}.
\end{proof}

\begin{Example}\label{E2.7}
Consider inequality~\eqref{E2.1.1} with the critical exponent $\lambda = p - 1$.
As in Example~\ref{E2.6}, we assume that
$b \in L_{\infty, loc} ({\mathbb R^n})$,
$q \in L_{\infty, loc} ({\mathbb R^n})$
is a non-negative function and, moreover, 
conditions~\eqref{2.1} and~\eqref{E2.1.2} are valid, 
where $k > -1$ and $l > k - p + 1$, respectively.

Let $u \ge 0$ be a nontrial solution of~\eqref{E2.1.1}.
In the case of $l \le p k$, applying Corollary~\ref{C2.8}, we have
\begin{equation}
	M (r; u)
	\ge
	e^{
		C r^{
			(l - k + p -1) / (p - 1)
		}
	}
	\label{E2.7.1}
\end{equation}
for all enough large $r$,
where the constant $C > 0$ does not depend of $r$ and $u$.
In turn, if $l > p k$, then~\eqref{E2.7.1} should be replaced by
$$
	M (r; u)
	\ge
	e^{
		C r^{
			(l + p) / p
		}
	}
$$
for all enough large $r$,
where the constant $C > 0$ does not depend of $r$ and $u$.
\end{Example}

\begin{Corollary}\label{C2.9}
Let~\eqref{2.1} be valid, where $k > -1$.
If there are real numbers $\theta > 1$ and $\sigma > 1$
such that~\eqref{T2.1.1} holds and
$$
	\int_{r_0}^\infty
	(r^{-k} q_\sigma (r))^{1 / (p - 1)}
	\,
	dr
	<
	\infty,
$$
then any solution of~\eqref{1.1}, \eqref{1.2} satisfies the estimate
$$
	\int_{
		M (r; u)
	}^\infty
	(g_\theta (t) t)^{- 1 / p}
	\,
	dt
	\ge
	C
	\left(
		\int_r^\infty
		(\xi^{- k} q_\sigma (\xi))^{1 / (p - 1)}
		\,
		d\xi
	\right)^{(p - 1) / p}
$$
for all sufficiently large $r$, 
where the constant $C > 0$ depends only on 
$n$, $p$, $\theta$, $\sigma$, $C_1$, $C_2$, $k$, and $\alpha$.
\end{Corollary}

\begin{proof}
We repeat the arguments given in the proof of Corollary~\ref{C2.5}
with Theorem~\ref{T2.1} replaced by Theorem~\ref{T2.4}.
\end{proof}

\begin{Example}\label{E2.8}
Assume that $u \ge 0$ is a nontrivial solution of inequality~\eqref{E2.1.1}, 
where $\lambda > p - 1$,
$b \in L_{\infty, loc} ({\mathbb R^n})$,
$q \in L_{\infty, loc} ({\mathbb R^n})$ 
is a non-negative function and, moreover, 
condition~\eqref{2.1} is fulfilled for some $k > -1$.

If~\eqref{E2.1.2} holds with $l < k - p + 1$,
then in accordance with Corollary~\ref{C2.9} we obtain
$$
	M (r; u)
	\le
	C
	r^{
		(l - k + p - 1) / (p - 1 - \lambda)
	}
$$
for all enough large $r$,
where the constant $C > 0$ does not depend of $r$ and $u$.

Now, let condition~\eqref{E2.2.1} be valid instead of~\eqref{E2.1.2},
where $\mu < 1 - p$. Then Corollary~\ref{C2.9} allows us to assert that
$$
	M (r; u)
	\le
	C
	\log^{
		(\mu + p - 1) / (p - 1 - \lambda)
	}
	r
$$
for all enough large $r$,
where the constant $C > 0$ does not depend of $r$ and $u$.
\end{Example}

One can easily verify that the estimates given in Examples~\ref{E2.5}--\ref{E2.8} are exact.

It does not matter for us that the right-hand side of~\eqref{2.1} is a power function.
Theorems~\ref{T2.1}--\ref{T2.4} remain precise for a wide class of functions $b$.

Assume that
\begin{equation}
	|b (x)|
	\le
	\beta
	|x|^k
	\log^m |x|
	\label{2.2}
\end{equation}
for almost all $x \in \Omega$,
where $\beta$, $k$, and $m$ are some constants 
and, moreover, either $k > -1$ or $k = -1$ and $m > 0$.

\begin{Corollary}\label{C2.6}
Let~\eqref{2.2} hold.
If there are real numbers $\theta > 1$ and $\sigma > 1$ such that
condition~\eqref{T2.1.1} is valid and
$$
	\int_{r_0}^\infty
	(
		r^{-k}
		\log^{-m} r
		\,
		q_\sigma (r)
	)^{1 / (p - 1)}
	\,
	dr
	=
	\infty
$$
then any solution of~\eqref{1.1}, \eqref{1.2} is trivial.
\end{Corollary}

\begin{proof}
From~\eqref{2.2}, we have
$f_{\sigma} (r) \ge \gamma r ^{- k - 1} \log^{- m} r \, q_\sigma (r)$
for all sufficiently large $r$,
where the constant $\gamma > 0$ depends only on $\beta$, $k$, $m$, and $\sigma$.
The proof is completed by applying of Theorem~\ref{T2.1}.
\end{proof}

\begin{Example}\label{E2.4}
In~\eqref{E2.1.1}, let
$b \in L_{\infty, loc} ({\mathbb R^n})$ 
satisfy condition~\eqref{2.2} and
$q \in L_{\infty, loc} ({\mathbb R^n})$
be a non-negative function such that
\begin{equation}
	q (x) \sim |x|^{k - p + 1}
	\quad
	\mbox{as } x \to \infty,
	\label{E2.4.1}
\end{equation}
i.e., in formula~\eqref{E2.1.2}, we take the critical exponent $l = k - p + 1$.

According to Corollary~\ref{C2.6}, if
\begin{equation}
	\lambda > p - 1
	\quad
	\mbox{and}
	\quad
	m \le p - 1,
	\label{E2.4.2}
\end{equation}
then any non-negative solution of inequality~\eqref{E2.1.1} is trivial.
As mentioned above, the first condition in~\eqref{E2.4.2} is exact.
Now, we show that the second condition is also exact.
Really, let
$$
	\lambda > p - 1
	\quad
	\mbox{and}
	\quad
	m > p - 1.
$$
By direct calculation, one can verify that
$$
	u (x)
	=
	(\log \max \{ |x|, r_0 \})^{
		(m - p + 1) / (\lambda - p + 1)
	}
$$
is a positive solution of~\eqref{E2.1.1} for enough large $r_0$, 
where
$q \in L_{\infty, loc} ({\mathbb R^n})$ 
is a non-negative function such that~\eqref{E2.4.1} holds 
and
$b \in L_{\infty, loc} ({\mathbb R^n})$
is a non-negative function satisfying the relation
$$
	b (x)
	\sim
	|x|^k
	\log^m |x|
	\quad
	\mbox{as } x \to \infty.
$$
\end{Example}

\section{Proof of Theorems~\ref{T2.1}--\ref{T2.4}}\label{Proof}

We need a well-known result concerning inequalities of the form
\begin{equation}
	\diver A (x, D u)
	\ge
	F (x, u, D u)
	\quad
	\mbox{in }
	\Omega_{R_0, R_1},
	\label{3.1}
\end{equation}
where 
$0 < R_0 < R_1 \le \infty$,
the function
$A : \Omega \times {\mathbb R}^n \to {\mathbb R}^n$
is the same as in~\eqref{1.1} and
$F : \Omega_{R_0, R_1} \times [0, \infty) \times {\mathbb R}^n \to {\mathbb R}$
satisfies the following conditions:
there exist a real number $\sigma > 1$ and locally bounded measurable functions 
$F_0 : [R_0, R_1) \times (0, \infty) \to [0,\infty)$
and
$F_1 : [R_0, R_1) \to [0, \infty)$
such that
$$
	F_0 (r, t-0) = F_0 (r, t)
	\quad
	\mbox{for all } R_0 < r < R_1, \:  t > 0,
$$
$$
	F_0 (r, t_1) \ge F_0 (r, t_2)
	\quad
	\mbox{for all } R_0 < r < R_1, \: t_1 \ge t_2 > 0
$$
and
\begin{equation}
	F (x, t, \zeta)
	\ge
	\sup_{
		r
		\in
		(|x| / \sigma, \sigma |x|)
		\cap
		(R_0, R_1)
	}
	F_0 (r, t)
	-
	|\zeta|^{p - 1}
	\inf_{
		r
		\in
		(|x| / \sigma, \sigma |x|)
		\cap
		(R_0, R_1)
	}
	F_1 (r)
	\label{3.3}
\end{equation}
for almost all
$x \in \Omega_{R_0, R_1}$
and for all
$t \in (0, \infty)$
and
$\zeta \in {\mathbb R}^n$.

We say that $u$ is a solution of~\eqref{3.1} if
$
	u
	\in
	W_p^1 (
		\Omega_{R_0, r}
	)
	\cap
	L_\infty (
		\Omega_{R_0, r}
	),
$
$
	A (x, D u)
	\in
	L_{p/(p-1)} (
		\Omega_{R_0, r}
	),
$
and
$
	F (x, u, D u)
	\in
	L_{p/(p-1)} (
		\Omega_{R_0, r}
	)
$
for any real number $r \in (R_0, R_1)$ and, moreover,
$$
	- \int_{
		\Omega_{R_0, R_1}
	}
	A (x, D u)
	D \varphi
	\, dx
	\ge
	\int_{
		\Omega_{R_0, R_1}
	}
	F (x, u, D u)
	\varphi
	\, dx
$$
for any non-negative function
$
	\varphi 
	\in 
	C_0^\infty (
		\Omega_{R_0, R_1}
	).
$

According to this definition, any solution of~\eqref{3.1} must be a non-negative function;
otherwise the right-hand side of the last inequality is not well-defined.

Let us denote
$
	\Gamma_{R_0, R_1}
	=
	B_{R_0, R_1}
	\cap
	\partial \Omega,
$
where
$
	B_{R_0, R_1}
	=
	\{
		x \in {\mathbb R}^n : R_0 < |x| < R_1
	\}.
$
The condition
\begin{equation}
	\left.
		u
	\right|_{
		\Gamma_{R_0, R_1}
	}
	=
	0
	\label{3.2}
\end{equation}
means that
$
	\varphi u
	\in
	{
		\stackrel{\rm \scriptscriptstyle o}{W}\!\!{}_p^1
		(
			\Omega_{R_0, R_1}
		)
	}
$
for any
$
	\varphi
	\in
	C_0^\infty 
	(
		B_{R_0, R_1}
	).
$

We shall assume that $S_r \cap \Omega \ne \emptyset$ for all $r \in (R_0, R_1)$.

\begin{Lemma}\label{L3.1}
Suppose that $0 < \beta < 1$ and $R_0 < \rho_0 < \rho_1 < R_1$ are some real numbers.
Also let $u$ be a solution of~\eqref{3.1}, \eqref{3.2} such that
${M (\cdot; u)}$ is a non-decreasing function on 
${(R_0, R_1)}$ satisfying the relation
$
	{M (R_0 + 0; u)} > 0.
$
If
$\beta^{1/2} {M (\rho_1; u)} \le {M (\rho_0; u)}$
and
$\sigma^2 \rho_0 \ge \rho_1$,
then
$$
	M (\rho_1; u) - M (\rho_0; u)
	\ge
	\gamma
	\min
	\left\{
		(\rho_1 - \rho_0)^{p / (p - 1)},
		\frac{
			\rho_1 - \rho_0
		}{
			\lambda^{1 / (p - 1)}
		}
	\right\}
	F_0^{1 / (p-1)} (s, \beta M (\rho_1; u))
$$
for all $s \in [\rho_1 / \sigma, \sigma \rho_0] \cap (R_0, R_1)$,
where
$$
	\lambda
	=
	\inf_{
		[\rho_1 / \sigma, \sigma \rho_0]
		\cap
		(R_0, R_1)
	}
	F_1
$$
and the constant $\gamma > 0$ depends only on
$n$, $p$, $C_1$, $C_2$, and $\beta$.
\end{Lemma}

The proof is given in paper~\cite[Lemma~3.1]{Me}.

\begin{Remark}\label{R3.1}
If $u$ is a solution of~\eqref{1.1}, \eqref{1.2}, 
then
$$
	M (r; u)
	=
	\esssup_{
		B_r
		\cap
		\Omega
	}
	u,
	\quad
	r \in (R_0, R_1),
$$
in accordance with the maximum principle.
Hence, ${M (\cdot; u)}$ is a non-decreasing function on ${(R_0, R_1)}$.
In addition, we have
${M (r; u)} = {M (r - 0; u)}$
for all 
$r \in {(R_0, R_1)}$.
\end{Remark}

\begin{Lemma}\label{L3.2}
Let $0 < R_0 < R_1 < \infty$ and, moreover,
$u$ be a solution of~\eqref{1.1}, \eqref{1.2} such that
$
	{M (R_0 + 0; u)} > 0
$
and
$\beta^{1/2} {M (R_1 - 0; u)} \le {M (R_0 + 0; u)}$
for some real number $0 < \beta < 1$.
Then
\begin{equation}
	M (R_1 - 0; u) - M (R_0 + 0; u)
	\ge
	\gamma
	\min
	\left\{
		(R_1 - R_0)^{p / (p - 1)},
		\frac{
			R_1 - R_0
		}{
			\mu^{1 / (p - 1)}
		}
	\right\}
	Q
	G,
	\label{L3.2.1}
\end{equation}
where $\gamma > 0$ is the constant of Lemma~$\ref{3.1}$,
$$
	Q
	=
	\essinf_{
		\Omega_{R_0, R_1}
	}
	q^{1 / (p-1)},
$$
$$
	G
	=
	\min_{
		[\beta^2 M (R_1 - 0; u), M (R_1 - 0; u)]
	}
	g^{1 / (p-1)},
$$
and
$$
	\mu
	=
	\esssup_{
		\Omega_{R_0, R_1}
	}
	|b|.
$$
\end{Lemma}

\begin{proof}
As mentioned in Remark~\ref{R3.1}, 
${M (\cdot; u)}$ is a non-decreasing function on the interval ${(R_0, R_1)}$. 
Also it can be seen that
$$
	M (R_1 - 0; u)
	=
	\esssup_{
		B_{R_1}
		\cap
		\Omega
	}
	u.
$$
In particular, $u (x) \le M (R_1 - 0; u)$ for almost all $x \in B_{R_1} \cap \Omega$.

Let us denote
$
	I
	=
	(\beta^2 M (R_1 - 0; u), \infty).
$
The function $u$ is a solution of~\eqref{3.1}, \eqref{3.2}, where
$$
	F (x, t, \zeta)
	=
	\chi_I (t)
	Q^{p - 1}
	G^{p - 1}
	-
	\mu
	|\zeta|^{p - 1}.
$$
Here, $\chi_I$ is the characteristic function of the set $I$, i.e.
$$
	\chi_I (t)
	=
	\left\{
		\begin{array}{ll}
			1,
			&
			t
			\in
			I,
			\\
			0,
			&
			t
			\not\in
			I.
		\end{array}
	\right.
$$
Putting further
$
	F_0 (r, t)
	=
	\chi_I (t)
	Q^{p - 1}
	G^{p - 1}
$
and
$
	F_1 (r) = \mu,
$
one can verify that~\eqref{3.3} is fulfilled for any $\sigma > 1$.
Thus, applying Lemma~\ref{L3.1}, we obtain
$$
	M (\rho_1; u) - M (\rho_0; u)
	\ge
	\gamma
	\min
	\left\{
		(\rho_1 - \rho_0)^{p / (p - 1)},
		\frac{
			\rho_1 - \rho_0
		}{
			\mu^{1 / (p - 1)}
		}
	\right\}
	Q
	G
$$
for all $R_0 < \rho_0 < \rho_1 < R_1$, whence~\eqref{L3.2.1} follows at once.

The proof is completed.
\end{proof}

\begin{Lemma}\label{L3.3}
Let $u$ be a solution of~\eqref{1.1}, \eqref{1.2}.
If $r_0 \le r_1 < r_2$, $\eta > 1$, and $\tau > 1$ are real numbers 
such that $\eta M (r_1 + 0; u) \le M (r_2; u)$, $r_2 \le \tau r_1$,
and ${M (r_1 + 0; u)} > 0$,
then at least one of the following two inequalities is valid:
\begin{equation}
	\int_{
		M (r_1 + 0; u)
	}^{
		M (r_2; u)
	}
	(g_\eta (t) t)^{-1 / p}
	\,
	dt
	\ge
	C
	\int_{r_1}^{r_2}
	q_\tau^{1 / p}
	(\xi)
	\,
	d\xi,
	\label{L3.3.1}
\end{equation}
\begin{equation}
	\int_{
		M (r_1 + 0; u)
	}^{
		M (r_2; u)
	}
	g_\eta^{-1 / (p - 1)}
	(t)
	\,
	dt
	\ge
	C
	\int_{r_1}^{r_2}
	(\xi f_\tau (\xi))^{1 / (p - 1)}
	\,
	d\xi,
	\label{L3.3.2}
\end{equation}
where the constant $C > 0$ depends only on $n$, $p$, $\eta$, $C_1$, and $C_2$.
\end{Lemma}

\begin{proof}
The function ${M (\cdot; u)}$ does not decrease on the interval ${[r_1, r_2]}$ and, moreover,
${M (r; u)} = {M (r - 0; u)}$
for all 
$r \in {(r_1, r_2]}$.
By induction, we construct the finite sequence
of real numbers $\rho_0 < \rho_1 < \ldots < \rho_k$.
Let us take $\rho_0 = r_2$. 
Now, assume that $\rho_i$ is already known. We put
\begin{equation}
	\rho_{i+1}
	=
	\inf
	\{
		\xi \in (r_1, \rho_i)
		:
		M (\xi; u) > \eta^{- 1 / 8} M (\rho_i; u)
	\}.
	\label{PL3.3.1}
\end{equation}
If $\rho_{i+1} = r_1$, then we take $k = i + 1$ and stop.
It is obvious that this procedure must terminate at a finite step.
Since ${M (\cdot; u)}$ is a semicontinuous function,
one can claim that
$
	\{
		\xi \in {(r_1, \rho_i)}
		:
		{M (\xi; u)} > {\eta^{- 1 / 8} M (\rho_i; u)}
	\}
	\ne
	\emptyset
$
for any $i \in \{ 0, \ldots, k - 1 \}$.
Thus, the right-hand side of~\eqref{PL3.3.1} is well defined.
In so doing, we have
\begin{equation}
	 \eta^{- 1 / 8} M (\rho_i; u) \le M (\rho_{i + 1} + 0; u),
	 \quad
	 i = 0, \dots, k - 1
	\label{PL3.3.2}
\end{equation}
and
\begin{equation}
	 M (\rho_{i + 1}; u) \le \eta^{- 1 / 8} M (\rho_i; u),
	 \quad
	 i = 0, \dots, k - 2.
	\label{PL3.3.3}
\end{equation}

For any $i \in \{ 0, \ldots, k - 1 \}$, according to Lemma~\ref{L3.2}, 
at least one of the following two estimates holds:
\begin{equation}
	M (\rho_i; u) - M (\rho_{i + 1} + 0; u)
	\ge
	\varkappa_1
	(\rho_i - \rho_{i + 1})^{p / (p - 1)}
	Q_i
	G_i,
	\label{PL3.3.4}
\end{equation}
\begin{equation}
	M (\rho_i; u) - M (\rho_{i + 1} + 0; u)
	\ge
	\frac{
		\varkappa_1
		(\rho_i - \rho_{i + 1})
		Q_i
		G_i
	}{
		\mu_i^{1 / (p - 1)}
	},
	\label{PL3.3.5}
\end{equation}
where
$$
	Q_i
	=
	\essinf_{
		\Omega_{\rho_{i + 1}, \rho_i}
	}
	q^{1 / (p-1)},
$$
$$
	G_i
	=
	\min_{
		[\eta^{- 1 / 2} M (\rho_i; u), M (\rho_i; u)]
	}
	g^{1 / (p-1)},
$$
$$
	\mu_i
	=
	\esssup_{
		\Omega_{\rho_{i + 1}, \rho_i}
	}
	|b|,
$$
and $\varkappa_1 > 0$ is some constant depending only on
$n$, $p$, $C_1$, $C_2$, and $\eta$.

By $\Xi_1$ we denote the set of integers $i \in \{ 0, \ldots, k - 1 \}$ 
for which~\eqref{PL3.3.4} is fulfilled.
Also let
$\Xi_2 = \{ 0, \ldots, k - 1 \} \setminus \Xi_1$.

At first, we assume that
\begin{equation}
	\sum_{i \in \Xi_1}
	(\rho_i - \rho_{i + 1})
	\ge
	\frac{
		r_2 - r_1
	}{
		2
	}.
	\label{PL3.3.6}
\end{equation}
Formula~\eqref{PL3.3.4} implies the relation
$$
	\left(
		\frac{
			M (\rho_i; u) - M (\rho_{i + 1} + 0; u)
		}{
			G_i
		}
	\right)^{(p - 1) / p}
	\ge
	\varkappa_1^{(p - 1) / p}
	(\rho_i - \rho_{i + 1})
	Q_i^{(p - 1) / p},
$$
whence, taking into account~\eqref{PL3.3.2}, we obtain
$$
	\frac{
		M (\rho_i; u) - \eta^{-1 / 8} M (\rho_i; u)
	}{
		G_i^{(p - 1) / p}
		M^{1 / p} (\rho_i; u)
	}
	\ge
	\varkappa_2
	(\rho_i - \rho_{i + 1})
	Q_i^{(p - 1) / p}
$$
for any $i \in \Xi_1$, where 
$\varkappa_2 = (1 - \eta^{-1 / 8})^{1 / p} \varkappa_1^{(p - 1) / p}$.
Combining this with the inequality
$$
	\int_{
		\eta^{-1 / 8} M (\rho_i; u)
	}^{
		M (\rho_i; u)
	}
	(g_\eta (t) t)^{-1 / p}
	\,
	dt
	\ge
	\frac{
		M (\rho_i; u) - \eta^{-1 / 8} M (\rho_i; u)
	}{
		G_i^{(p - 1) / p}
		M^{1 / p} (\rho_i; u)
	},
$$
one can conclude that
$$
	\int_{
		\eta^{-1 / 8} M (\rho_i; u)
	}^{
		M (\rho_i; u)
	}
	(g_\eta (t) t)^{-1 / p}
	\,
	dt
	\ge
	\varkappa_2
	(\rho_i - \rho_{i + 1})
	Q_i^{(p - 1) / p}
$$
for any $i \in \Xi_1$.
Summing the last estimate over all $i \in \Xi_1 \cap \{0, \ldots, k - 2 \}$, we have
$$
	\sum_{
		i \in \Xi_1 \cap \{0, \ldots, k - 2 \}
	}
	\int_{
		\eta^{-1 / 8} M (\rho_i; u)
	}^{
		M (\rho_i; u)
	}
	(g_\eta (t) t)^{-1 / p}
	\,
	dt
	\ge
	\varkappa_2
	\essinf_{
		\Omega_{r_1, r_2}
	}
	q^{1 / p}
	\sum_{
		i \in \Xi_1 \cap \{0, \ldots, k - 2 \}
	}
	(\rho_i - \rho_{i + 1}).
$$
By~\eqref{PL3.3.3}, this implies that
\begin{equation}
	\int_{
		M (\rho_{k - 1}; u)
	}^{
		M (\rho_0; u)
	}
	(g_\eta (t) t)^{-1 / p}
	\,
	dt
	\ge
	\varkappa_2
	\essinf_{
		\Omega_{r_1, r_2}
	}
	q^{1 / p}
	\sum_{
		i \in \Xi_1 \cap \{0, \ldots, k - 2 \}
	}
	(\rho_i - \rho_{i + 1}).
	\label{PL3.3.7}
\end{equation}
If $k - 1 \not\in \Xi_1$, then combining~\eqref{PL3.3.6} and~\eqref{PL3.3.7}, 
we immediately obtain~\eqref{L3.3.1}.
Let ${k - 1} \in \Xi_1$. Putting $i = k - 1$ in formula~\eqref{PL3.3.4}, 
we obviously have
$$
	\left(
		\frac{
			M (\rho_{k - 1}; u) - M (\rho_k + 0; u)
		}{
			G_{k - 1}
		}
	\right)^{(p - 1) / p}
	\ge
	\varkappa_1^{(p - 1) / p}
	(\rho_{k - 1} - \rho_k)
	Q_{k - 1}^{(p - 1) / p}.
$$
At the same time,
${M (\rho_{k - 1}; u)} \le {\eta^{1 / 8} M (\rho_k + 0; u)}$
according to~\eqref{PL3.3.2}; therefore, one can assert that
\begin{equation}
	\frac{
		\eta^{1 / 8} M (\rho_k + 0; u) - M (\rho_k + 0; u)
	}{
		G_{k - 1}^{(p - 1) / p}
		M^{1 / p} (\rho_k + 0; u)
	}
	\ge
	\varkappa_3
	(\rho_{k - 1} - \rho_k)
	Q_{k - 1}^{(p - 1) / p},
	\label{PL3.3.8}
\end{equation}
where
$\varkappa_3 = (\eta^{1 / 8} - 1)^{1 / p} \varkappa_1^{(p - 1) / p}$.
Since
$$
	\int_{
		M (\rho_k + 0; u)
	}^{
		\eta^{1 / 8} M (\rho_k + 0; u)
	}
	(g_\eta (t) t)^{-1 / p}
	\,
	dt
	\ge
	\frac{
		\eta^{1 / 8} M (\rho_k + 0; u) - M (\rho_k + 0; u)
	}{
		\eta^{p / 8}
		G_{k - 1}^{(p - 1) / p}
		M^{1 / p} (\rho_k + 0; u)
	},
$$
formula~\eqref{PL3.3.8} implies the estimate
$$
	\int_{
		M (\rho_k + 0; u)
	}^{
		\eta^{1 / 8} M (\rho_k + 0; u)
	}
	(g_\eta (t) t)^{-1 / p}
	\,
	dt
	\ge
	\varkappa_4
	(\rho_{k - 1} - \rho_k)
	Q_{k - 1}^{(p - 1) / p},
$$
where
$\varkappa_4 = \eta^{- p / 8} \varkappa_3$.
Summing this with~\eqref{PL3.3.7}, we obtain
$$
	\int_{
		M (r_1 + 0; u)
	}^{
		M (r_2; u)
	}
	(g_\eta (t) t)^{-1 / p}
	\,
	dt
	\ge
	\varkappa_5
	\essinf_{
		\Omega_{r_1, r_2}
	}
	q^{1 / p}
	\sum_{
		i \in \Xi_1
	}
	(\rho_i - \rho_{i + 1}),
$$
where $\varkappa_5 = \min \{ \varkappa_2, \varkappa_4 \} / 2$.
Finally, combining the last inequality and~\eqref{PL3.3.6}, we again derive~\eqref{L3.3.1}.

Now, assume that~\eqref{PL3.3.6} is not valid. Thus, we have
\begin{equation}
	\sum_{i \in \Xi_2}
	(\rho_i - \rho_{i + 1})
	\ge
	\frac{
		r_2 - r_1
	}{
		2
	}.
	\label{PL3.3.9}
\end{equation}
Let $i \in \Xi_2$.
From~\eqref{PL3.3.5}, it follows that
$$
	\frac{
		M (\rho_i; u) - M (\rho_{i + 1} + 0; u)
	}{
		G_i
	}
	\ge
	\frac{
		\varkappa_1
		(\rho_i - \rho_{i + 1})
		Q_i
	}{
		\mu_i^{1 / (p - 1)}
	},
$$
whence, taking into account the relation
\begin{equation}
	\int_{
		M (\rho_{i + 1} + 0; u)
	}^{
		M (\rho_i; u)
	}
	g_\eta^{-1 / (p - 1)}
	(t)
	\,
	dt
	\ge
	\frac{
		M (\rho_i; u) - M (\rho_{i + 1} + 0; u)
	}{
		G_i
	},
	\label{PL3.3.10}
\end{equation}
we obtain
\begin{equation}
	\int_{
		M (\rho_{i + 1} + 0; u)
	}^{
		M (\rho_i; u)
	}
	g_\eta^{-1 / (p - 1)}
	(t)
	\,
	dt
	\ge
	\frac{
		\varkappa_1
		(\rho_i - \rho_{i + 1})
		Q_i
	}{
		\mu_i^{1 / (p - 1)}
	}.
	\label{PL3.3.11}
\end{equation}

Summing~\eqref{PL3.3.11} over all $i \in \Xi_2$, one can conclude that
$$
	\int_{
		M (r_1 + 0; u)
	}^{
		M (r_2; u)
	}
	g_\eta^{-1 / (p - 1)}
	(t)
	\,
	dt
	\ge
	\frac{
		\varkappa_1
		\essinf_{
			\Omega_{r_1, r_2}
		}
		q^{1 / (p - 1)}
	}{
		\esssup_{
			\Omega_{r_1, r_2}
		}
		|b|^{1 / (p - 1)}
	}
	\sum_{i \in \Xi_2}
	(\rho_i - \rho_{i + 1}).
$$
By~\eqref{PL3.3.9}, this immediately implies~\eqref{L3.3.2}.

The proof is completed.
\end{proof}

\begin{Lemma}\label{L3.4}
Let $u$ be a solution of~\eqref{1.1}, \eqref{1.2}.
If $r_0 \le r_1 < r_2$, $\eta > 1$, and $\tau > 1$ are real numbers 
such that $M (r_2; u) \le \eta M (r_1 + 0; u)$, $\tau^{1 / 2} r_1 \le r_2$,
and ${M (r_1 + 0; u)} > 0$,
then estimate~\eqref{L3.3.2} holds,
where the constant $C > 0$ depends only on $n$, $p$, $\eta$, $\tau$, $C_1$, and $C_2$.
\end{Lemma}

\begin{proof}
Consider the maximal integer $k$ satisfying the condition $\tau^{k / 2} r_1 \le r_2$.
We put $\rho_i = \tau^{-i / 2} r_2$, $i = 0, \ldots, k - 1$, and $\rho_k = r_1$.
From Lemma~\ref{L3.2}, it follows that at least one of 
inequalities~\eqref{PL3.3.4}, \eqref{PL3.3.5} is valid 
for any $i \in \{ 0, \ldots, k - 1 \}$.
If~\eqref{PL3.3.4} holds for some $i \in \{ 0, \ldots, k - 1 \}$, then we have
$$
	\frac{
		M (\rho_i; u) - M (\rho_{i + 1} + 0; u)
	}{
		G_i
	}
	\ge
	\varkappa_1
	(\rho_i - \rho_{i + 1})^{p / (p - 1)}
	Q_i.
$$
Since $\rho_i - \rho_{i + 1} \ge (1 - \tau^{-1 / 2}) \rho_i$, this implies the estimate
$$
	\frac{
		M (\rho_i; u) - M (\rho_{i + 1} + 0; u)
	}{
		G_i
	}
	\ge
	\varkappa_6
	(\rho_i - \rho_{i + 1})
	\rho_i^{1 / (p - 1)}
	Q_i,
$$
where the constant $\varkappa_6 > 0$ depends only on
$n$, $p$, $\eta$, $\tau$, $C_1$, and $C_2$. 
Consequently, taking into account~\eqref{PL3.3.10} and the fact that
$$
	\int_{
		\rho_{i + 1}
	}^{
		\rho_i
	}
	(\xi f_\tau (\xi))^{1 / (p - 1)}
	\,
	d\xi
	\le
	(\rho_i - \rho_{i + 1})
	\rho_i^{1 / (p - 1)}
	Q_i,
$$
we obtain
\begin{equation}
	\int_{
		M (\rho_{i + 1} + 0; u)
	}^{
		M (\rho_i; u)
	}
	g_\eta^{-1 / (p - 1)}
	(t)
	\,
	dt
	\ge
	\varkappa_6
	\int_{
		\rho_{i + 1}
	}^{
		\rho_i
	}
	(\xi f_\tau (\xi))^{1 / (p - 1)}
	\,
	d\xi.
	\label{PL3.4.1}
\end{equation}
In turn, if~\eqref{PL3.3.5} is fulfilled for some $i \in \{ 0, \ldots, k - 1 \}$, then
$$
	\frac{
		M (\rho_i; u) - M (\rho_{i + 1} + 0; u)
	}{
		G_i
	}
	\ge
	\frac{
		\varkappa_1
		(\rho_i - \rho_{i + 1})
		Q_i
	}{
		\mu_i^{1 / (p - 1)}
	},
$$
whence in accodance with~\eqref{PL3.3.10} and the evident inequality
$$
	\int_{
		\rho_{i + 1}
	}^{
		\rho_i
	}
	(\xi f_\tau (\xi))^{1 / (p - 1)}
	\,
	d\xi
	\le
	\frac{
		(\rho_i - \rho_{i + 1})
		Q_i
	}{
		\mu_i^{1 / (p - 1)}
	}
$$
formula~\eqref{PL3.4.1} follows again.

To complete the proof, it remains to sum~\eqref{PL3.4.1} over all $i \in \{ 0, \ldots, k - 1 \}$.
\end{proof}

\begin{Lemma}\label{L3.5}
Let $u$ be a solution of~\eqref{1.1}, \eqref{1.2}.
If $r_0 \le r_1 < r_2$, $\theta > 1$, and $\sigma > 1$ are real numbers 
such that $\theta M (r_1 + 0; u) \le M (r_2; u)$, $r_2 \ge \sigma r_1$,
and ${M (r_1 + 0; u)} > 0$, then
\begin{equation}
	\int_{
		M (r_1 + 0; u)
	}^{
		M (r_2; u)
	}
	(g_\theta (t) t)^{-1 / p}
	\,
	dt
	\ge
	C
	\left(
		\int_{r_1}^{r_2}
		(\xi f_\sigma (\xi))^{1 / (p - 1)}
		\,
		d\xi
	\right)^{(p - 1) / p},
	\label{L3.5.1}
\end{equation}
where the constant $C > 0$ depends only on $n$, $p$, $\theta$, $\sigma$, $C_1$, and $C_2$.
\end{Lemma}

\begin{proof}
We denote $\tau = \sigma^{1/2}$ and $\eta = \theta^{1/2}$. 
Take the maximal integer $k$ satisfying the condition $\tau^{k / 2} r_1 \le r_2$.
As in the proof of Lemma~\ref{L3.4}, we put 
$\rho_i = \tau^{- i / 2} r_2$, $i = 0, \ldots, k - 1$, and $\rho_k = r_1$.

According to Lemmas~\ref{L3.3} and~\ref{L3.4}, for any $i \in \{ 0, \ldots, k - 1 \}$ 
at least one of the following two inequalities holds:
\begin{equation}
	\int_{
		M (\rho_{i + 1} + 0; u)
	}^{
		M (\rho_i; u)
	}
	(g_\eta (t) t)^{-1 / p}
	\,
	dt
	\ge
	\varkappa_7
	\int_{
		\rho_{i + 1}
	}^{
		\rho_i
	}
	q_\tau^{1 / p}
	(\xi)
	\,
	d\xi,
	\label{PL3.5.1}
\end{equation}
\begin{equation}
	\int_{
		M (\rho_{i + 1} + 0; u)
	}^{
		M (\rho_i; u)
	}
	g_\eta^{-1 / (p - 1)}
	(t)
	\,
	dt
	\ge
	\varkappa_7
	\int_{
		\rho_{i + 1}
	}^{
		\rho_i
	}
	(\xi f_\tau (\xi))^{1 / (p - 1)}
	\,
	d\xi,
	\label{PL3.5.2}
\end{equation}
where the constant $\varkappa_7 > 0$ depends only on $n$, $p$, $\eta$, $\tau$, $C_1$, and $C_2$.

By $\Xi_1$ we mean the set of integers $i \in \{ 0, \ldots, k - 1 \}$ 
for which~\eqref{PL3.5.1} is fulfilled.
Also put $\Xi_2 = \{ 0, \ldots, k - 1 \} \setminus \Xi_1$ .

At first, let
\begin{equation}
	\sum_{i \in \Xi_1}
	\int_{
		\rho_{i + 1}
	}^{
		\rho_i
	}
	(\xi f_\sigma (\xi))^{1 / (p - 1)}
	\,
	d\xi
	\ge
	\frac{1}{2}
	\int_{
		r_1
	}^{
		r_2
	}
	(\xi f_\sigma (\xi))^{1 / (p - 1)}
	\,
	d\xi.
	\label{PL3.5.3}
\end{equation}
Summing~\eqref{PL3.5.1} over all $i \in \Xi_1$, we have
\begin{equation}
	\int_{
		M (r_1 + 0; u)
	}^{
		M (r_2; u)
	}
	(g_\eta (t) t)^{-1 / p}
	\,
	dt
	\ge
	\varkappa_7
	\sum_{i \in \Xi_1}
	\int_{
		\rho_{i + 1}
	}^{
		\rho_i
	}
	q_\tau^{1 / p}
	(\xi)
	\,
	d\xi.
	\label{PL3.5.4}
\end{equation}
For any $i \in \{ 0, \ldots, k - 1 \}$ there exists $\xi_i \in (\rho_{i+1}, \rho_i)$ such that
$$
	\int_{
		\rho_{i + 1}
	}^{
		\rho_i
	}
	q_\tau^{1 / p}
	(\xi)
	\,
	d\xi
	\ge
	(\rho_i - \rho_{i + 1})
	q_\tau^{1 / p}
	(\xi_i).
$$
Thus,
\begin{align*}
	\left(
		\sum_{i \in \Xi_1}
		\int_{
			\rho_{i + 1}
		}^{
			\rho_i
		}
		q_\tau^{1 / p}
		(\xi)
		\,
		d\xi
	\right)^{p / (p - 1)}
	&
	\ge
	\sum_{i \in \Xi_1}
	\left(
		\int_{
			\rho_{i + 1}
		}^{
			\rho_i
		}
		q_\tau^{1 / p}
		(\xi)
		\,
		d\xi
	\right)^{p / (p - 1)}
	\\
	&
	\ge
	\sum_{i \in \Xi_1}
	(\rho_i - \rho_{i + 1})^{p / (p - 1)}
	q_\tau^{1 / (p - 1)}
	(\xi_i),
\end{align*}
whence, taking into account the fact that
\begin{align*}
	(\rho_i - \rho_{i + 1})^{p / (p - 1)}
	q_\tau^{1 / (p - 1)}
	(\xi_i)
	&
	\ge
	(1 - \tau^{-1 / 2})^{1 / (p - 1)}
	(\rho_i - \rho_{i + 1})
	(
		\rho_i
		q_\tau
		(\xi_i)
	)^{1 / (p - 1)}
	\\
	&
	\ge
	(1 - \tau^{-1 / 2})^{1 / (p - 1)}
	\int_{
		\rho_{i + 1}
	}^{
		\rho_i
	}
	(\xi f_\sigma (\xi))^{1 / (p - 1)}
	\,
	d\xi
\end{align*}
for any $i \in \Xi_1$, we obtain
$$
	\left(
		\sum_{i \in \Xi_1}
		\int_{
			\rho_{i + 1}
		}^{
			\rho_i
		}
		q_\tau^{1 / p}
		(\xi)
		\,
		d\xi
	\right)^{p / (p - 1)}
	\ge
	(1 - \tau^{-1 / 2})^{1 / (p - 1)}
	\sum_{i \in \Xi_1}
	\int_{
		\rho_{i + 1}
	}^{
		\rho_i
	}
	(\xi f_\sigma (\xi))^{1 / (p - 1)}
	\,
	d\xi.
$$
The last estimate, \eqref{PL3.5.3}, and~\eqref{PL3.5.4} immediately imply~\eqref{L3.5.1}.

Now, assume that~\eqref{PL3.5.3} is not valid.
In this case, we obviously have
\begin{equation}
	\sum_{i \in \Xi_2}
	\int_{
		\rho_{i + 1}
	}^{
		\rho_i
	}
	(\xi f_\sigma (\xi))^{1 / (p - 1)}
	\,
	d\xi
	\ge
	\frac{1}{2}
	\int_{
		r_1
	}^{
		r_2
	}
	(\xi f_\sigma (\xi))^{1 / (p - 1)}
	\,
	d\xi.
	\label{PL3.5.5}
\end{equation}
For every $i \in \Xi_2$ inequality~\eqref{PL3.5.2} holds. Consequently, 
\begin{equation}
	\int_{
		M (r_1 + 0; u)
	}^{
		M (r_2; u)
	}
	g_\eta^{-1 / (p - 1)}
	(t)
	\,
	dt
	\ge
	\varkappa_7
	\sum_{i \in \Xi_2}
	\int_{
		\rho_{i + 1}
	}^{
		\rho_i
	}
	(\xi f_\tau (\xi))^{1 / (p - 1)}
	\,
	d\xi.
	\label{PL3.5.6}
\end{equation}
Let us show that
\begin{equation}
	\left(
		\int_{
			M (r_1 + 0; u)
		}^{
			M (r_2; u)
		}
		(g_\theta (t) t)^{-1 / p}
		\,
		dt
	\right)^{p / (p - 1)}
	\ge
	\varkappa_8
	\int_{
		M (r_1 + 0; u)
	}^{
		M (r_2; u)
	}
	g_\eta^{-1 / (p - 1)}
	(t)
	\,
	dt,
	\label{PL3.5.7}
\end{equation}
where the constant $\varkappa_8 > 0$ depends only on $p$ and $\theta$.

Really, take the maximal integer $l$ satisfying the condition 
$\eta^{l / 2} M (r_1 + 0; u) \le M (r_2; u)$.
We put 
$m_i = \eta^{- i / 2} M (r_2; u)$,
$i = 0, \ldots, l - 1$,
and
$m_l = M (r_1 + 0; u)$.

For any $i \in \{ 0, \ldots, l - 1 \}$ there exists $t_i \in (m_{i+1}, m_i)$ such that
$$
	\int_{
		m_{i + 1}
	}^{
		m_i
	}
	(g_\theta (t) t)^{-1 / p}
	\,
	dt
	\ge
	(m_i - m_{i + 1})
	(g_\theta (t_i) t_i)^{-1 / p}.
$$
Hence, we have
\begin{align*}
	\left(
		\int_{
			M (r_1 + 0; u)
		}^{
			M (r_2; u)
		}
		(g_\theta (t) t)^{-1 / p}
		\,
		dt
	\right)^{p / (p - 1)}
	&
	=
	\left(
		\sum_{i = 0}^{l - 1}
		\int_{
			m_{i + 1}
		}^{
			m_i
		}
		(g_\theta (t) t)^{-1 / p}
		\,
		dt
	\right)^{p / (p - 1)}
	\\
	&
	\ge
	\sum_{i = 0}^{l - 1}
	\left(
		\int_{
			m_{i + 1}
		}^{
			m_i
		}
		(g_\theta (t) t)^{-1 / p}
		\,
		dt
	\right)^{p / (p - 1)}
	\\
	&
	\ge
	\sum_{i = 0}^{l - 1}
	(m_i - m_{i + 1})^{p / (p - 1)}
	(g_\theta (t_i) t_i)^{-1 / (p - 1)}.
\end{align*}
Since
\begin{align*}
	(m_i - m_{i + 1})^{p / (p - 1)}
	(g_\theta (t_i) t_i)^{-1 / (p - 1)}
	&
	\ge
	(1 - \eta^{-1 / 2})^{1 / (p - 1)}
	(m_i - m_{i + 1})
	g_\theta^{-1 / (p - 1)} (t_i)
	\\
	&
	\ge
	(1 - \eta^{-1 / 2})^{1 / (p - 1)}
	\int_{
		m_{i + 1}
	}^{
		m_i
	}
	g_\eta^{-1 / (p - 1)} (t)
	\,
	dt
\end{align*}
for all $i \in \{ 0, \ldots, l - 1 \}$,
this implies the estimate
$$
	\left(
		\int_{
			M (r_1 + 0; u)
		}^{
			M (r_2; u)
		}
		(g_\theta (t) t)^{-1 / p}
		\,
		dt
	\right)^{p / (p - 1)}
	\ge
	(1 - \eta^{-1 / 2})^{1 / (p - 1)}
	\sum_{i = 0}^{l - 1}
	\int_{
		m_{i + 1}
	}^{
		m_i
	}
	g_\eta^{-1 / (p - 1)} (t)
	\,
	dt
$$
from which~\eqref{PL3.5.7} follows at once.

To complete the proof, 
it remains to combine~\eqref{PL3.5.5}, \eqref{PL3.5.6}, and~\eqref{PL3.5.7}.
\end{proof}

\begin{proof}[Proof of Theorem~$\ref{T2.1}$]
Let $u$ be a nontrivial solution of~\eqref{1.1}, \eqref{1.2}.
There obviously exists a real number $r_1 > r_0$ such that $M (r_1; u) > 0$. 
By the maximum principle, ${M (\cdot; u)}$ is a non-decreasing 
function on the interval $(r_0, \infty)$; therefore, 
${M (r_1 + 0; u)} \ge {M (r_1; u)} > 0$. 

At first, assume that
\begin{equation}
	\lim_{r \to \infty}
	M (r; u)
	=
	M (\infty; u)
	<
	\infty.
	\label{PT2.1.1}
\end{equation}
Taking the real number $r_1$ large enough, we obtain
${M (r_1 + 0; u)} \le \theta {M (r; u)}$
for all $r > r_1$.
Thus,
\begin{equation}
	\int_{
		M (r_1 + 0; u)
	}^{
		M (\infty; u)
	}
	g_\theta^{-1 / (p - 1)}
	(t)
	\,
	dt
	\ge
	C
	\int_{r_1}^\infty
	(\xi f_\sigma (\xi))^{1 / (p - 1)}
	\,
	d\xi
	\label{PT2.1.2}
\end{equation}
by Lemma~\ref{L3.4}, where the constant $C > 0$ depends 
only on $n$, $p$, $\theta$, $\sigma$, $C_1$, and $C_2$.
This contradicts condition~\eqref{T2.1.2}.

Now, let
\begin{equation}
	\lim_{r \to \infty}
	M (r; u)
	=
	\infty,
	\label{PT2.1.3}
\end{equation}
then Lemma~\ref{L3.5} allows us to assert that
$$
	\int_{
		M (r_1 + 0; u)
	}^{
		\infty
	}
	(g_\theta (t) t)^{-1 / p}
	\,
	dt
	\ge
	C
	\left(
		\int_{r_1}^\infty
		(\xi f_\sigma (\xi))^{1 / (p - 1)}
		\,
		d\xi
	\right)^{(p - 1) / p}
$$
with some constant $C > 0$ dependind only on $n$, $p$, $\theta$, $\sigma$, $C_1$, and $C_2$.
In turn, this contradicts conditions~\eqref{T2.1.1} and~\eqref{T2.1.2}.

The proof is completed.
\end{proof}

\begin{proof}[Proof of Theorem~$\ref{T2.2}$]
Let $u$ be a nontrivial solution of~\eqref{1.1}, \eqref{1.2}.
As in the proof of Theorem~\ref{T2.1}, we take a real number $r_1 > r_0$ 
such that ${M (r_1; u)} > 0$. 
If~\eqref{PT2.1.1} holds, then Lemma~\ref{L3.4} implies inequality~\eqref{PT2.1.2} 
which contradicts~\eqref{T2.1.2}.
Consequently, one can assume that~\eqref{PT2.1.3} is fulfilled.

According to Lemma~\ref{L3.5},
$$
	\int_{
		M (r_1 + 0; u)
	}^{
		M (r; u)
	}
	(g_\theta (t) t)^{-1 / p}
	\,
	dt
	\ge
	C
	\left(
		\int_{r_1}^r
		(\xi f_\sigma (\xi))^{1 / (p - 1)}
		\,
		d\xi
	\right)^{(p - 1) / p}
$$
for all sufficiently large $r$,
where the constant $C > 0$ depends only on $n$, $p$, $\theta$, $\sigma$, $C_1$, and $C_2$.
In so doing, from~\eqref{T2.1.2} and~\eqref{T2.2.1}, it follows that
$$
	\int_{r_1}^r
	(\xi f_\sigma (\xi))^{1 / (p - 1)}
	\,
	d\xi
	\ge
	\frac{1}{2}
	\int_{r_0}^r
	(\xi f_\sigma (\xi))^{1 / (p - 1)}
	\,
	d\xi
$$
and
$$
	\int_1^{
		M (r; u)
	}
	(g_\theta (t) t)^{-1 / p}
	\,
	dt
	\ge
	\frac{1}{2}
	\int_{
		M (r_1 + 0; u)
	}^{
		M (r; u)
	}
	(g_\theta (t) t)^{-1 / p}
	\,
	dt
$$
for all sufficiently large $r$.

Thus, combining the last three estimates, we complete the proof.
\end{proof}

\begin{proof}[Proof of Theorem~$\ref{T2.3}$]
Let $u$ be a nontrivial solution of~\eqref{1.1}, \eqref{1.2}.
We take a real number $r_1 > r_0$ such that ${M (r_1; u)} > 0$.
Also assume that $r > r_1$ is a real number satisfying the condition $\sigma r_1 < r$ 
and, moreover, $k$ is the maximal integer for which $\sigma^{k/2} r_1 \le r$.
We put $\rho_i = \sigma^{-i/2} r$, $i = 0, \ldots, k - 1$, and $\rho_k = r_1$.

By Lemmas~\ref{L3.3} and~\ref{L3.4}, for any $i \in \{ 0, \ldots, k - 1 \}$ 
at least one of the following two inequalities is valid:
$$
	\int_{
		M (\rho_{i + 1} + 0; u)
	}^{
		M (\rho_i; u)
	}
	(g_\theta (t) t)^{-1 / p}
	\,
	dt
	\ge
	C
	\int_{
		\rho_{i + 1}
	}^{
		\rho_i
	}
	q_\sigma^{1 / p}
	(\xi)
	\,
	d\xi,
$$
$$
	\int_{
		M (\rho_{i + 1} + 0; u)
	}^{
		M (\rho_i; u)
	}
	g_\theta^{-1 / (p - 1)}
	(t)
	\,
	dt
	\ge
	C
	\int_{
		\rho_{i + 1}
	}^{
		\rho_i
	}
	(\xi f_\sigma (\xi))^{1 / (p - 1)}
	\,
	d\xi,
$$
where the constant $C > 0$ depends only on $n$, $p$, $\theta$, $\sigma$, $C_1$, and $C_2$.
Hence,
\begin{align*}
	&
	\int_{
		M (\rho_{i + 1} + 0; u)
	}^{
		M (\rho_i; u)
	}
	(g_\theta (t) t)^{-1 / p}
	\,
	dt
	+
	\int_{
		M (\rho_{i + 1} + 0; u)
	}^{
		M (\rho_i; u)
	}
	g_\theta^{-1 / (p - 1)}
	(t)
	\,
	dt
	\\
	&
	\quad
	\ge
	C
	\int_{
		\rho_{i + 1}
	}^{
		\rho_i
	}
	\min
	\{
		(\xi f_\sigma (\xi))^{1 / (p - 1)},
		\,
		q_\sigma^{1 / p}
		(\xi)
	\}
	\,
	d\xi
\end{align*}
for any $i \in \{ 0, \ldots, k - 1 \}$.
Summing the last estimate over all $i \in \{ 0, \ldots, k - 1 \}$ , we obtain
\begin{align}
	&
	\int_{
		M (r_1 + 0; u)
	}^{
		M (r; u)
	}
	(g_\theta (t) t)^{-1 / p}
	\,
	dt
	+
	\int_{
		M (r_1 + 0; u)
	}^{
		M (r; u)
	}
	g_\theta^{-1 / (p - 1)}
	(t)
	\,
	dt
	\nonumber
	\\
	&
	\quad
	\ge
	C
	\int_{r_1}^r
	\min
	\{
		(\xi f_\sigma (\xi))^{1 / (p - 1)},
		\,
		q_\sigma^{1 / p}
		(\xi)
	\}
	\,
	d\xi.
	\label{PT2.3.1}
\end{align}
On the other hand, condition~\eqref{T2.3.1} allows us to assert that
\begin{align*}
	\int_{r_1}^r
	\min
	\{
		(\xi f_\sigma (\xi))^{1 / (p - 1)},
		\,
		q_\sigma^{1 / p}
		(\xi)
	\}
	\,
	d\xi
	\ge
	{}
	&
	\frac{1}{2}
	\int_{r_0}^r
	\min
	\{
		(\xi f_\sigma (\xi))^{1 / (p - 1)},
		\,
		q_\sigma^{1 / p}
		(\xi)
	\}
	\,
	d\xi
	\\
	&
	+
	\frac{1}{C}
	\int_{
		M (r_1 + 0; u)
	}^1
	(g_\theta (t) t)^{-1 / p}
	\,
	dt
	\\
	&
	+
	\frac{1}{C}
	\int_{
		M (r_1 + 0; u)
	}^1
	g_\theta^{-1 / (p - 1)}
	(t)
	\,
	dt
\end{align*}
if the real number $r$ is large enough.
Combining this with~\eqref{PT2.3.1}, we complete the proof.
\end{proof}

\begin{proof}[Proof of Theorem~$\ref{T2.4}$]
Let $u$ be a solution of~\eqref{1.1}, \eqref{1.2}.
If ${M (\cdot; u)}$ is a bounded function on the interval $(r_0, \infty)$,
then inequality~\eqref{T2.4.1} is obviously valid for all sufficiently large $r$
since, in this case, the right-hand side of~\eqref{T2.4.1} tends to zero as $r \to \infty$,
whereas the left-hand side is bounded below by a positive real number 
for all $r \in (r_0, \infty)$.

In turn, if ${M (\cdot; u)}$ is not a bounded function, then~\eqref{PT2.1.3} holds.
We remind that ${M (\cdot; u)}$ is a non-decreasing function on $(r_0, \infty)$ 
according to the maximum principle.
Thus, estimate~\eqref{T2.4.1} follows immediately from Lemma~\ref{L3.5}.

The proof is completed.
\end{proof}

\begin{Remark}\label{R3.2}
From the arguments given above, it can be seen that Theorems~\ref{T2.1}--\ref{T2.4} 
remain hold for solutions of the problem
$$
	\left\{
		\begin{aligned}
		&
		{\rm div} \, A (x, D u)
		+
		b (x) |D u|^{p - 1}
		\ge
		q (x) g (u)
		\quad
		\mbox{in }
		\Omega_{r_0, \infty},
		\\
		&
		\left.
			u
		\right|_{
			\Gamma_{r_0, \infty}
		}
		=
		0,
		\end{aligned}
	\right.
$$
where
$
	\Omega_{r_0, \infty}
	=
	\{
		x \in \Omega : r_0  < |x|
	\}
$
and
$
	\Gamma_{r_0, \infty}
	=
	\{
		x \in \partial \Omega : r_0  < |x|
	\}.
$
In this case, in the hypotheses of Theorems~\ref{T2.1}--\ref{T2.3},
we must additionaly assume that ${M (\cdot; u)}$ 
is a non-decreasing function on the interval $(r_0, \infty)$.
\end{Remark}

\end{document}